\documentclass[final, leqno]{amsart}

\usepackage{amsmath,amsthm}
\usepackage {latexsym}
\usepackage{amssymb}

\newcommand{\eps}{{\varepsilon}}
\newcommand{\R}{{\mathbb R}}

\newcommand{\sph}{{\mathbb S}}

\newcommand{\meas}{{\mathcal M}}
\newcommand{\B}{{\mathcal B}}
\newcommand{\W}{{\mathcal W}}

\newcommand{\les}{\lesssim}
\newcommand{\gtr}{\gtrsim}

\newcommand{\Kato}{{\mathcal K}}
\newcommand{\txt}{\textstyle}

\newcommand{\Lap}{\Delta}

\newcommand{\Res}{R_0^+}
\newcommand{\Rplus}{R^+}
\newcommand{\Rminus}{R^-}
\newcommand{\vp}{\varphi}

\newcommand{\la}{\langle}
\newcommand{\ra}{\rangle}
\newcommand{\1}{{\mathbf 1}}
\def\norm[#1][#2]{\|#1\|_{#2}}
\def\bignorm[#1][#2]{\big\|#1\big\|_{#2}}
\def\Bignorm[#1][#2]{\Big\|#1\Big\|_{#2}}
\def\japanese[#1]{\langle #1 \rangle}
\def\Im[#1]{{\rm Im}(#1)}
\def\Re[#1]{{\rm Re}(#1)}

\newtheorem{theorem}{Theorem}
\newtheorem{lemma}[theorem]{Lemma}

\newtheorem{corollary}[theorem]{Corollary}

\newtheorem{proposition}[theorem]{Proposition}

\theoremstyle{remark}
\newtheorem{remark}{Remark}

\begin{document}

\title[Dispersive Estimates for Measure-Valued Potentials]
{Dispersive Estimates for Schr\"odinger Operators with Measure-Valued
Potentials in $\R^3$}

\date{August 3, 2011}

\author{Michael\ Goldberg}
\thanks{This work is supported in part by NSF grant DMS-1002515.}
\address{Department of Mathematics, University of Cincinnati,
Cincinnati, OH 45221-0025}
\email{Michael.Goldberg@uc.edu}

\maketitle

\begin{abstract}
We prove dispersive estimates for the linear Schr\"odinger evolution associated
to an operator $-\Delta + V$ in $\R^3$, where the potential is a signed measure
with fractal dimension at least 3/2.
\end{abstract}

\section{Introduction}

The dispersive properties of the free Schr\"odinger semigroup $e^{it\Delta}$
as a map between $L^p(\R^n)$ and its dual space are well understood,
thanks to Plancherel's identity (or more generally the Spectral Theorem) 
in the case $p=2$, and Fourier inversion in the case $p=1$.  On one endpoint,
the $L^2$ conservation laws extend readily to any self-adjoint perturbation
$H = -\Delta + V$ taking the place of $-\Delta$ as the infinitesimal generator.
Our goal in this paper is to establish a corresponding $L^1 \mapsto L^\infty$
estimate in three dimensions for a class of short-range potentials $V(x)$ that
include measures as admissible local singularities.

Measure-valued potentials are quite common in one dimension; the operator
$-\frac{d^2}{dx^2} + c\delta_0$ is often the subject of exercises in a first
quantum mechanics course.  In higher dimensions there are several plausible
generalizations of this example.  Dispersive estimates are known in the case
where $V(x)$ consists of a finite collection of point masses in
$\R^3$~\cite{DaPiTe06}.  In these results the spaces $L^1$ and $L^\infty$ are
modified by a set of local weights because the domain of the associated
Schr\"odinger operator consists of functions that vanish at each point mass. 
Here we preserve the idea of the potential describing an
infinitesimally thin barrier and
show that dispersive estimates are valid in unweighted $L^p(\R^3)$
when $V(x)$ is supported on a compact two-dimensional surface
$\Sigma \subset \R^3$.  In fact we will consider all
compactly supported fractal measures of sufficiently high dimension.
For many purposes the threshold dimension is 1 (in $\R^n$ it would be $n-2$)
so that multiplication by $V$ is compact relative to the Laplacian.  
We are forced to increase the threshold dimension to 3/2 in the
proof of the Schr\"odinger dispersive estimate in order to use the best available
Fourier restriction theorems.  

In this paper, a compactly supported signed measure $\mu$ is called
{\em $\alpha$-dimensional} if it satisfies
\begin{equation} \label{eq:dimension}
|\mu|(B(x,r)) \leq C_\mu r^{\alpha} \ \ {\rm for\ all}\ r>0\ {\rm and}\ 
x \in \R^3
\end{equation}
Nontrivial $\alpha$-dimensional measures exist for any $\alpha \in [0,3]$.
We also characterize potentials in terms of the {\em global Kato norm}, defined
on signed measures in $\R^3$ by the quantity
\begin{equation} \label{eq:Kato}
\norm[\mu][\Kato] = \sup_{y\in\R^3} \int_{\R^3} \frac{|\mu|(dx)}{|x-y|}
\end{equation}

Every element with finite global Kato norm is a 1-dimensional measure with
$C_\mu \leq \norm[\mu][\Kato]$, by comparing $|x-y|^{-1}$ to the characteristic
function of a ball.  The converse is not quite true, however the Kato class contains
all compactly supported measures of dimension $\alpha > 1$.  We will examine
this relationship more carefully in Proposition~\ref{prop:dimensionKato}.

\begin{remark}
Kato's work~\cite{Ka72} is more closely associated with the local uniform
integrability condition~\eqref{eq:localKato}; the first true norm of this type
(integrating over $|x-y| < 1$ only) is due to Schechter~\cite{Sc71}.
We follow the naming convention in
Rodnianski-Schlag~\cite{RoSc04} where the global Kato norm is applied to
dispersive estimates in $\R^3$.
\end{remark}

For the free Schr\"odinger equation in $\R^n$, the standard dispersive bound is
\begin{equation} \label{eq:freedispersive}
\norm[e^{it\Lap}f][\infty] \le (4\pi|t|)^{-n/2}\norm[f][1]
\end{equation}
In three dimensions this inequality is stable under small perturbations of the
Laplacian.  Once the negative part of $V$ is sufficiently large
(e.g. $\norm[V_-][\Kato] > 4\pi$) it becomes possible for $H = -\Delta + V$
to acquire one or more bound states that evolve in place without time-decay
according to a law $e^{-itH}\psi_j = e^{-it\lambda_j}\psi_j$.
We wish to show that with the exception of bound
states, the Schr\"odinger propagator of $H$ still satisfies an
estimate of the form~\eqref{eq:freedispersive}.

Our main result imposes an additional spectral assumption that all eigenvalues
of $H$ be strictly negative, and that there is no resonance at zero.  In this
context a resonance occurs when the equation
\begin{equation*}
\psi + (-\Delta - (\lambda \pm i0))^{-1}V\psi = 0
\end{equation*}
has nontrivial solutions belonging to the weighted space
$\japanese[x]^{s}L^2(\R^3)$ for each $s > \frac12$ but not to $L^2$ itself.
Such functions also solve $H\psi = \lambda\psi$, however the lack of
square-integrability gives resonances different spectral properties from a true
eigenvalue.  Forbidding
eigenvalues and resonances at zero is a common practice, as it is known
that the continuous part of the linear Schr\"odinger evolution may have
leading-order decay of $|t|^{-1/2}$ if zero is not a regular point of the
spectrum~\cite{ErSc04},~\cite{Ya05}.  The necessity of a spectral assumption
over the interval $\lambda > 0$ is uncertain but it is included here for the
sake of simplicity.

\begin{theorem} \label{thm:main}
Let $V$ be a compactly supported signed measure on $\R^3$ of dimension
$d > \frac32$.  If the Schr\"odinger operator $-\Delta + V$ has no resonance
at zero and no eigenvalues at any $\lambda \geq 0$, then the dispersive
estimate
\begin{equation} \label{eq:dispersive}
\norm[e^{-it(-\Delta+V)}P_{ac}f][\infty] \les |t|^{-3/2}\norm[f][1]
\end{equation}
holds for every $f \in L^1(\R^3)$.  The symbol $P_{ac}$ denotes projection
onto the continuous spectrum of $-\Delta + V$.

The dispersive estimate is also valid if, for a fixed $d > \frac32$, $V$ can be
expressed as the Kato-norm limit of compactly supported $d$-dimensionsal
measures.
\end{theorem}

\begin{remark}
During the course of the proof we
demonstrate that resonances cannot exist at any $\lambda > 0$
(Lemma~\ref{lem:Agmon}), and that embedded eigenvalues also cannot exist
provided the dimension of $V$ is greater than 2 (Proposition~\ref{prop:absence}).
The uniform resolvent bounds that are central to the argument also suffice to
prove the absence of singular continuous spectrum by applying Theorem~XIII.20
of~\cite{ReSi4}.
\end{remark}

Dispersive estimates with a time decay rate of $|t|^{-3/2}$ were found by
Rauch~\cite{Ra78} and Jensen-Kato~\cite{JeKa79} for initial data belonging
to weighted $L^2(\R^3)$.  The first statement of type~\eqref{eq:dispersive}
was proved by Journ\'e-Soffer-Sogge~\cite{JoSoSo91} for potentials satisfying
both $\hat{V} \in L^1$ and $|V(x)| \les |x|^{-7-\epsilon}$.  Here the effects
of the perturbation are computed directly onto the Schr\"odinger propagator
using Duhamel's formula.  Several authors have since refined the older spectral
methods to reproduce~\eqref{eq:dispersive} with less
restrictive conditions on the potential (\cite{Ya95}, \cite{RoSc04},
\cite{GoSc04a}, \cite{Go06b}, \cite{BeGo11}).  This has been particularly
effective in  three dimensions, thanks to a convenient expression for the
resolvent of the Laplacian as an integral operator.  Progress along these lines
in other dimensions (with time decay $|t|^{-n/2}$ for estimates on $\R^n$)
 can be found in~\cite{Sc05a}, \cite{CaCuVo09}, \cite{ErGr10}, along with
the results in~\cite{JoSoSo91}.

We follow the same procedure as in~\cite{BeGo11}, where the dispersive
estimate is derived from an integrability property of a family of operators
that describes the difference between the free and perturbed spectral measures.
The desired integrability follows in turn from a Wiener $L^1$ inversion
theorem involving Fourier analysis of operator-valued functions on the real line.

There are two main difficulties with extending previous work to the class of
measure-valued potentials.  The first is to verify that that multiplication by $V$
has small form-bound relative to the Laplacian so that one can speak freely
regarding the self-adjointness of $-\Delta + V$ and its essential spectrum.
The second is to identify function spaces on which multiplication by $V$ is well
defined (which excludes any $L^p(\R^3)$) and the resolvent of the Laplacian
has suitable asymptotics.  Most of the analysis takes place in $L^2(V)$ for this
reason.  The embedding $\dot{H}^1(\R^3) \subset L^2(V)$ plays a key role
mediating between the two types of operators and insuring that the end result
is still translation-invariant.

Section~\ref{sec:selfadjoint} addresses the properties of $V$ as a quadratic
form over $\dot{H}^1(\R^3)$ and spells out basic relations between
this Sobolev space and the global Kato norm.  These results are not surprising
but we are unaware of a careful treatment in the literature.  The proof of
Theorem~\ref{thm:main} unfolds over the course of
Section~\ref{sec:dispersive}.  We recall the reduction 
argument and abstract Wiener theorem from~\cite{BeGo11} then show that each one
of its hypotheses are satisfied for the class of potentials under consideration.
The high energy resolvent bounds (Theorem~\ref{thm:resolvent}) may shed light on
other scattering phenomena beyond the scope of the current paper. 

\section{Self-Adjointness} \label{sec:selfadjoint}

For any potential which is not a bounded function of $x$ there are well known
difficulties identifying the domain of $-\Delta + V$ and its adjoint operator.
We can take advantage of the KLMN theorem~\cite[Theorem X.17]{ReSi2}
to produce a unique self-adjoint operator with the correct quadratic form on
$\dot{H}^1(\R^3)$  provided $V$ satisfies the form bound
\begin{equation}\label{eq:formbound}
\Big|\int_{\R^3}|\vp(x)|^2\,dV\Big| \le a\norm[\vp][\dot{H}^1]^2 
  + b\norm[\vp][L^2]^2
\end{equation}
for some $a < 1$.  It will suffice to assume that $V$
satisfies the``local Kato condition"
\begin{equation} \label{eq:localKato}
\lim_{r \to 0^+} \sup_{y\in\R^3} \int_{|x-y|<r} |x-y|^{-1}\,|V|(dx) \ =\ 0.
\end{equation}
Measures that are $\alpha$-dimensional for some $\alpha > 1$ automatically
satisfy~{\eqref{eq:localKato} with an explicit modulus of continuity as $r$ 
approaches zero.

\begin{proposition} \label{prop:dimensionKato}
Suppose $\mu$ is an $\alpha$-dimensional measure, $\alpha > 1$, with support
in the ball $B(0,2^M)$.  Then $\mu \in\Kato$ with the global and local estimates
\begin{align*}
&\norm[\mu][\Kato] \les \frac{C_\mu}{\alpha-1} 2^{(\alpha-1)M} \\
{\it and}\quad &\sup_{y\in\R^3}\int_{|x-y|<r} \frac{|\mu|(dx)}{|x-y|}
\les \frac{C_\mu}{\alpha-1}r^{\alpha-1} \quad{\it for\ all}\ r>0.
\end{align*}
Consequently, if $V$ can be approximated in $\Kato$ by a sequence
of measures $\mu_j$ with dimension $\alpha_j > 1$, then
$V$ satisfies~\eqref{eq:localKato}. 
\end{proposition}

\begin{proof}
For each point $y \in B(0,2^{M+1})$, 
\begin{equation} \label{eq:dimensionKato}
\begin{aligned} \int_{\R^3} \frac{|\mu|(dx)}{|x-y|}
&\leq \sum_{k=-\infty}^\infty 2^{-k}|\mu|(B(y,2^k)) \\
&\les C_\mu 2^{(M+2)(\alpha-1)}
 \bigg(\frac{1}{1-2^{(1-\alpha)}} + 1\bigg)
\end{aligned}
\end{equation}
by estimating $|\mu|(B(y,2^k)) \leq C_\mu 2^{\alpha\max(k,M+2)}$.
To integrate over the region of finite radius $r$, the sum in~\eqref{eq:dimensionKato}
is taken over $k \leq \lceil \log r \rceil$ instead.

If $|y| > 2^{M+1}$ then the integral in~\eqref{eq:dimensionKato} is easily
bounded by $2|y|^{-1}|\mu|(B(0,2^M))$ by observing that $|x-y| \sim |y|$
within the  support of $\mu$.

Convergence of  $\mu_j$ in the global Kato norm forces the collection of
functions $\eta_j(r) = \sup_y \int_{|x-y|< r} |x-y|^{-1}\, d|\mu_j|$
to converge uniformly in $r$.
The property $\lim_{r\to 0} \eta_j(r) = 0$
is preserved by uniform convergence.
\end{proof}

The class of functions $V(x)$ (i.e. absolutely continuous measures $V(x)\,dx$) 
defined by property~\eqref{eq:localKato} is considered at length in~\cite{Si82}.
It is suggested there that singular measures satisfying~\eqref{eq:localKato}
may be approximated by a bounded function via convolution with smooth
mollifiers.  While this approach is indeed useful we emphasize that convergence
in the Kato norm generally fails because of the placement of absolute values.
The weak convergence argument that takes its place is detailed below.

Note that~\eqref{eq:localKato} implies that $\Delta^{-1}V$ is a uniformly
continuous function.  Moreover the entire family $\Delta^{-1}(V\omega)$
is equicontinuous, where $\omega$ ranges over the bounded measurable
functions of unit norm.  Both these claims are proved as part of
Lemma~\ref{lem:compactness} in the next section. Then we have the estimates
\begin{align*}
\norm[\tau_z\Delta^{-1}V f][L^{\infty}(V)] &\les
\norm[V][\Kato] \norm[f][L^{\infty}(V)]  \quad {\rm for\ all}\ z \in \R^3\\
\norm[(\1 - \tau_z)\Delta^{-1}V f][L^{\infty}(V)] &\les
o(1)\norm[f][L^{\infty}(V)]  \qquad {\rm as}\ |z| \to 0.
\end{align*}
The vanishing rate $o(1)$ depends on the specific profile of $V$ but is
independent of the choice of $f\in L^\infty(V)$.

By duality, and the fact that translations commute with $\Delta^{-1}$,
the same operator estimates hold for $L^1(V)$ as well.  Applying the Schur
test to the integral kernels of these operators
extends the result to all $L^p(V)$, $1 \le p \le \infty$.
This gives an embedding of $\dot{H}^1(\R^3)$ into $L^2(V)$ because by
a $TT^*$ argument it suffices to show that $(-\Delta)^{-1}V$ is a bounded
map from $L^2(V)$ to itself.

The function space $L^2(V)$ is not preserved by translations, however
the embedded subspace $\dot{H}^1(\R^3) \subset L^2(V)$ is translation invariant.
The fact that the action of translations on this subspace is continuous with
respect to the $L^2(V)$ norm is also verified by a $TT^*$ argument.  Let
$T = ({\bf 1} - \tau_z)|\nabla|^{-1}: L^2(\R^3) \to L^2(V)$.
Then
\begin{align*}
\norm[T][] &= \big(\norm[TT^*V][L^2(V) \to L^2(V)]\big)^{1/2} \\
&= \Big(\bignorm[\big((\1-\tau_z)+(\1-\tau_{-z})\big)\Delta^{-1}V
  ][L^2(V) \to L^2(V)]\Big)^{1/2} \\
&= o(1)
\end{align*}
with the end result that $\norm[f - \tau_z f][L^2(V)] 
\le o(1)\norm[f][\dot{H}^1]$ for all $f \in \dot{H}^1(\R^3)$.

For the purposes of~\eqref{eq:formbound} there is not much distance between
$V$ and its close translates because
\begin{align*}
\int_{\R^3} |\vp(x)|^2\,dV - \int_{\R^3} |\vp(x)|^2\, dV(x-z)
&= \int_{\R^3} \big(|\vp(x)|^2 - |\vp(x+z)|^2\big)\,dV \\
&= \int_{\R^3} \big(\vp(x) - \vp(x+z)\big)\bar{\vp}(x)\,dV \\
&\quad + \int_{\R^3} \vp(x+z)\big(\bar{\vp}(x) - \bar{\vp}(x+z)\big) \,dV \\
\le\ &\norm[(\1 - \tau_{-z})\vp][L^2(V)]\big(\norm[\vp][L^2(V)] + \norm[\tau_{-z}\vp][L^2(V)]\big)\\
\le\ &o(1) \norm[\vp][\dot{H}^1]^2
\end{align*}

Let $V^r(x)$ be the quantity $r^{-3}\int_{B(x,r)}\,dV$, which represents the
``average value" of $V$ over a ball radius $r$.  For fixed $r>0$, $V^r(x)$ is a 
continuous function bounded by $r^{-2}\norm[V][\Kato]$.  Then by splitting
$V = (V - V^r) + V^r$ and averaging the above inequality over all $|z| < r$
we see that
\begin{align*}
\Big|\int_{\R^3}|\vp(x)|^2\,dV\Big| &\le o(1)\norm[\vp][\dot{H}^1]^2 
+ \Big|\int_{\R^3} |\vp(x)|^2\,dV^r\Big| \\
&\le o(1)\norm[\vp][\dot{H}^1]^2 
  + r^{-2}\norm[V][\Kato]\norm[\vp][L^2]^2
\end{align*}
so the coefficient on $\norm[\vp][\dot{H}^1]$ becomes smaller than 1
provided $r$ is sufficiently close to zero.

\section{The Dispersive Estimate} \label{sec:dispersive}

The proof of the dispersive estimate for the Schr\"odinger operator
$-\Delta + V$ follows the same road-map and technical machinery as 
in~\cite{BeGo11}.  First one represents the propagator $e^{it(-\Delta+V)}$
as an integral over the spectral measure, which is expressed in terms of
resolvents via the Stone formula.  Start with the expression
\begin{equation} \label{eq:Stone}
e^{-itH}P_{ac} f = \frac{1}{2\pi i}\int_0^\infty e^{-it\lambda}
[\Rplus_V(\lambda) - \Rminus_V(\lambda)]f \,d\lambda,
\end{equation}
where $R_V^\pm(\lambda) := (H -(\lambda\pm i0))^{-1}$ are the perturbed
resolvents.  The relationship between $R_V^\pm(\lambda)$ and the corresponding
free resolvent $R_0^\pm(\lambda) = (-\Delta - (\lambda \pm i0))^{-1}$ is given
by the multiplicative identities
\begin{equation*}
R_V^\pm(\lambda) = (I + R_0^\pm(\lambda)V)^{-1}R_0^\pm(\lambda)
= R_0^\pm(\lambda)(I + VR_0^\pm(\lambda))^{-1}.
\end{equation*}

When~\eqref{eq:Stone} is subjected to a change of variable
$\lambda \to \lambda^2$ and integration by parts, the end result is
\begin{align*}
e^{-itH}P_{ac}&f
= \frac{1}{\pi i} \int_{-\infty}^\infty e^{-it\lambda^2}\lambda 
\Rplus_V(\lambda^2) f\,d\lambda\\
&= \frac{1}{2\pi t} \int_{-\infty}^\infty e^{-it\lambda^2}
\frac{d}{d\lambda}\Rplus_V(\lambda^2) f\,d\lambda\\
&= \frac{1}{2\pi t}\int_{-\infty}^\infty e^{-it\lambda^2}
\big(I + \Res(\lambda^2)V\big)^{-1}\frac{d}{d\lambda}\Big[\Res(\lambda^2)
\Big] \big(I+V\Res(\lambda^2)\big)^{-1}f\, d\lambda.
\end{align*}
We have made a slight shift in notation here, setting
$\Res(\lambda^2) = (-\Delta - (\lambda + i0)^2)^{-1}$
to account for the fact that $(\lambda + i0)^2 = (\lambda^2 - i0)$ when $\lambda < 0$.

The explicit formula for the free resolvent kernel in three dimensions is
\begin{equation} \label{eq:freeres}
\Rplus_0(\lambda^2)(x,y) = (4\pi|x-y|)^{-1} e^{i\lambda|x-y|}.
\end{equation}
Apply Parseval's identity to the last integral in $\lambda$, taking 
$e^{-it\lambda^2} \frac{d}{d\lambda}\big[\Res(\lambda^2)\big]$ 
to be one of the factors.  It is the Fourier transform of a bounded function
in all variables, with upper bound controlled by $|t|^{-1/2}$.

Our remaining task is to show that $(I + V\Rplus_0(\lambda^2))^{-1}f$ is the
Fourier transform of a measure on $\R^{1+3}$ whose total
variation norm is bounded by $\norm[f][1]$.  This is done by applying an
operator-valued Wiener $L^1$ Inversion Theorem~\cite[Theorem~3]{BeGo11}, 
and taking care to recognize where signed measures occur in lieu of integrable
functions. The dispersive estimate then follows by integration in absolute value.  

To set the background for the Wiener theorem, let $X$ be a Banach space
and $\W_X$ the space of bounded linear maps $T: X \to L^1(\R;X)$ with
associated norm
\begin{equation} \label{eq:W_Xnorm}
\norm[T][\W_X] = \sup_{\norm[f][X] = 1} \int_R \norm[Tf(\rho)][X]\,d\rho.
\end{equation}
This becomes an algebra under the product
\begin{equation} \label{eq:multiplication}
S*Tf(\rho) = \int_\R S\big(Tf(\sigma)\big)(\rho-\sigma)\, d\sigma
\end{equation}
and we use $\overline{\W}_X$ to denote the unital extension of $\W_X$.
\begin{remark}
In the case where there exists a family of bounded linear operators
$S(\rho): X \to X$ satisfying $S(\rho)f = Sf(\rho)$, the product formula
can be restated as a convolution
\begin{equation*}
S*T (\rho)f = \int_\R S(\rho-\sigma)Tf(\sigma)\, d\sigma.
\end{equation*}
More generally the associated ``cross-section''operators $S(\rho)$ may be
unbounded at each $\rho \in \R$. 
One such example with $X = L^1(\R)$ is to fix an integrable function $\eta$
and set $Sf(\rho) = f(\rho)\eta(x)$.
\end{remark}

The Fourier transform of an element $T \in \W_X$ is computed by its action
on test functions in $X$,
\begin{equation*}
\hat{T}(\lambda)f = \int_\R e^{-i\lambda\rho} Tf(\rho)\, d\rho.
\end{equation*}
The family of operators $\hat{T}(\lambda)$ are bounded uniformly by
$\norm[T][W_X]$, continuous in $\lambda$ with respect to the strong operator
topology, and converge strongly to zero as $|\lambda| \to \infty$.
The Fourier transform of the identity element in $\overline{\W}_X$ is
$\hat{\1}(\lambda) = I$.

Products in $\overline{\W}_X$ correspond to pointwise composition of
operators on the Fourier transform side, so an element
$T \in \overline{\W}_X$ cannot be invertible unless $\hat{T}(\lambda)$
is invertible in $\B(X)$ for each $\lambda$.  With two extra assumptions
this pointwise condition is also sufficient.

\begin{theorem}[{\cite[Theorem~3]{BeGo11}}] \label{thm:Wiener}
Suppose $T$ is an element of $\W_X$ satisfying the
properties
\renewcommand{\theenumi}{(C\arabic{enumi})}
\renewcommand{\labelenumi}{\theenumi}
\begin{enumerate}
\item \label{translation} $\lim\limits_{\delta \to 0} 
\norm[T(\rho) - T(\rho-\delta)][\W_X] = 0$.
\item \label{locality}  $\lim\limits_{R \to \infty}
\norm[\chi_{|\rho| \ge R} T][\W_X] = 0$.
\end{enumerate}
If $I + \hat{T}(\lambda)$ is an invertible element of $\B(X)$ for every
$\lambda \in \R$, then ${\mathbf 1} + T$ possesses an inverse in 
$\overline{\W}_X$ of the form ${\mathbf 1} + S$.

In fact it is only necessary for some finite power $T^N \in \W_X$ (using the
definition of products in $\W_X$ given by~\eqref{eq:multiplication}) to satisfy
the translation-continuity condition~{\rm \ref{translation}} rather than $T$ itself.
\end{theorem}

The proof is constructive, and it is important to note that the inverse of
${\mathbf 1} + T$  has norm controlled by the following quantities:
$\norm[T][\W_X]$, $\sup_{\lambda} 
\norm[(I + \hat{T}(\lambda))^{-1}][X\to X]  =: \alpha$,
the value of $R$ for which the norm in~\ref{locality} is smaller
than $1/(K\alpha)$, the exponent $N$, and the value of
$\delta$ for which the norm in~\ref{translation} is smaller than
$1/K$ when applied to $T^N$.  The auxilliary constant $K$ is determined by a
choice of cutoff functions used during the construction.

\begin{proposition} \label{prop:parameters}
Any subset $U \subset \W_X$ for which there is uniform control
over these five parameters will admit a uniform bound 
$\sup_{T \in U}\norm[({\mathbf 1} + T) ^{-1}][\overline{\W}_X] \leq C < \infty$.
\end{proposition}

For the application to Schr\"odinger dispersive bounds we would like to choose
$X$ to be the space $\meas$ of finite complex Borel measures on $\R^3$ and
$\hat{T}(\lambda) = V\Res(\lambda^2)$.  There is a slight technical obstruction
because the family of operators $T(\rho)$ have a distribution kernel
\begin{equation*}
K(\rho, x, y) = \frac{V(x)}{4\pi|x-y|} \delta_0(\rho+|x-y|).
\end{equation*}
Given a measure $\mu \in \meas$, the image $Tf$ is a measure on $\R^4$ whose total
variation satisfies $\norm[T\mu][] \le (4\pi)^{-1}\norm[V][\Kato] \norm[\mu][\meas]$
but it is not guaranteed to belong to $L^1(\R; \meas)$.  We work instead with
a regularized version of $T$ obtained by cutting off $\hat{T}(\lambda)$ to finite
support.  Let $\eta$ be a standard cutoff function on the line, and define
$T_L$ so that $\hat{T}_L(\lambda) = \eta (\lambda/L)\hat{T}(\lambda) =
\eta(\lambda/L)V\Res(\lambda^2)$.

For each $L > 0$ the integral kernel associated to $T_L$ is given by
\begin{equation} \label{eq:K_L}
K_L(\rho, x, y) = L\frac{V(x)}{4\pi|x-y|} \check{\eta}(L(\rho+|x-y|)).
\end{equation}
Therefore at a fixed value of $\rho$ we have a bound
\begin{equation*}
\int_{\R^3} |T_L\mu(\rho,\,\cdot\,)| \leq \frac{L}{4\pi}\norm[\check{\eta}][\sup]
\iint \frac{|V|(dx) \, |\mu|(dy)}{|x-y|} 
\leq \frac{L}{4\pi} \norm[V][\Kato] \norm[\mu][\meas].
\end{equation*}
Now that $T_L\mu$ is seen to be an $\meas$-valued function, one can also
check its $L^1$ norm by integrating
\begin{equation} \label{eq:T_LinW_M}
\begin{aligned}
\int_\R \norm[T_L\mu(\rho,\,\cdot\,)][\meas]\, d\rho
&\leq \int_{\R^7} d\rho \, L \big|\check{\eta}(L(\rho+|x-y|))\big|
 \frac{|V|(dx)}{4\pi|x-y|} |\mu|(dy) \\
&\leq \frac{\norm[\check{\eta}][1]}{4\pi}
 \int_{\R^6} \frac{|V|(dx)}{|x-y|} |\mu|(dy) \\
&\leq \frac{\norm[\check{\eta}][1]}{4\pi} \norm[V][\Kato] \norm[\mu][\meas].
\end{aligned}
\end{equation}
This demonstrates that each $T_L \in \W_\meas$, with $\norm[T_L][]$
bounded independently of $L$.  In order to prove Theorem~\ref{thm:main}
it suffices to show that
\begin{equation*}
\limsup_{L \to \infty} \norm[({\mathbf 1} + T_L)^{-1}][\overline{\W}_\meas]
 < \infty
\end{equation*}
as this will guarantee that $({\mathbf 1} + T)^{-1}f$ is a finite measure on
$\R^{1+3}$ by taking a distributional limit.  Proposition~\ref{prop:parameters}
provides a clear path for obtaining uniform estimates.

For the majority of the discussion we will assume that $V$ is compactly
supported and has dimension $d > \frac32$.  The list of modifications to
accomodate Kato-norm limits of such potentials is given at the conclusion.

Already there is a uniform bound for $\norm[T_L][\W_\meas]$, the next step is
to determine $\alpha$ by establishing a norm bound for
$(I + \eta(\lambda/L)V\Res(\lambda^2))^{-1}: \meas \to \meas$ 
that is uniform over $\lambda \in \R$, $L \ge L_0$.  There are separate
arguments for low/intermediate and high energy.  The lower-energy estimates
are based on compactness and absence of embedded eigenvalues.
The high-energy analysis is a decay estimate for oscillatory integrals
of the type encountered in Fourier restriction operators.  The technical work
encountered in this step will then make it easy to set values for the remaining
three parameters ($R$, $N$, and $\delta$).

Observe that the family of operators $V\Res(\lambda^2)$ is norm-continuous
with respect to $\lambda$ via the estimate
\begin{equation} \label{eq:continuity}
\begin{aligned}
\norm[V\Res(\lambda_1^2) - V\Res(\lambda_2^2)][\meas \to \meas]
&\leq \norm[V][\meas] \sup_{x,y\in \R^3} \Big|\frac{e^{i \lambda_1|x-y|}
 - e^{i \lambda_2|x-y|}}{4\pi |x-y|} \Big| \\
&\leq\norm[V][\meas] \frac{|\lambda_1-\lambda_2|}{4\pi}.
\end{aligned}
\end{equation}
Then the norm of $(I + V\Res(\lambda^2))^{-1}$ is a continuous function of
$\lambda$, and it is bounded on any finite interval $\lambda \in [-L_0, L_0]$
provided the operator $I + V\Res(\lambda^2)$ is invertible for each $\lambda$.
The Fredholm Alternative argument behind pointwise (in $\lambda$) invertability
is standard, however its details need to be checked in the function spaces
under consideration.

\begin{lemma} \label{lem:compactness}
Suppose $V$ is a compactly supported measure satisfying~\eqref{eq:localKato}.
Then for each $\lambda \in \R$, the operator $V\Res(\lambda^2):\meas\to\meas$
is compact.
\end{lemma}
\begin{proof}
It is easier to show that the operator $\Rminus_0(\lambda^2)V$ acts compactly
on the space of bounded functions in $\R^3$.  The stated result follows by
duality.

Let $g$ be any bounded measurable function with $\sup_{x\in\R^3}|g(x)| \leq 1$.  Then 
$|\Rminus_0(\lambda^2)Vg(x)| \leq (4\pi)^{-1}\norm[V][\Kato]$, so the image of the
unit ball is bounded as expected.  To verify equicontinuity, examine the difference
\begin{equation*}
|\Rminus_0(\lambda^2)Vg(x_1) - \Rminus_0(\lambda^2)Vg(x_2)|
\leq \int_{\R^3} |g(y)| \bigg|\frac{e^{-i\lambda|x_1-y|}}{4\pi|x_1-y|}
 - \frac{e^{-i\lambda|x_2-y|}}{4\pi|x_2-y|}\bigg|\,|V|(dy)
\end{equation*}
Fix a value of $r >0$ such that $\int_{|x-y|<2r} |x-y|^{-1} |V|(dy) < \eps$
for every $x \in \R^3$.  Assuming $|x_1 - x_2| < \frac{r}{2}$, the integral
splits into the regions $|y-x_1| < r$ and $|y-x_1| > r$.
In the former region there is little cancellation so the integral is maximized by 
\begin{equation*}
\int_{|y-x_1| < r} (4\pi|x_1-y|)^{-1}\,|V|(dy) + 
\int_{|y-x_2| < 2r} (4\pi|x_2-y|)^{-1}\,|V|(dy) < (2\pi)^{-1}\eps.
\end{equation*}
In the latter region the Mean value Theorem places a bound
\begin{align*}
\bigg|\frac{e^{-i\lambda|x_1-y|}}{4\pi|x_1-y|}
 - \frac{e^{-i\lambda|x_2-y|}}{4\pi|x_2-y|}\bigg|
&\leq \frac{1}{2\pi} \max(|\lambda|, {\txt \frac{2}{|y-x_1|}}) 
 \frac{|x_1 - x_2|}{|y-x_1|}\\
&< (2\pi)^{-1}\max(|\lambda|, 2r^{-1})\frac{|x_1 - x_2|}{|y-x_1|}
\end{align*}
where we have used the geometric property $|y-x_2| > \frac12|y-x_1|$.
It follows that $|\Rminus_0(\lambda^2)Vg(x_1) - \Rminus_0(\lambda^2)Vg(x_2)|
< C\eps$ provided $|x_1 - x_2| < \min(r, |\lambda|^{-1})\eps$.

Furthermore, for all $x$ well outside the support of $V$, there is the decay
estimate $|\Rminus_0(\lambda^2)Vg(x)| < 2|x|^{-1}\norm[V][\meas]$.
Compactness of the operator $\Rminus_0(\lambda^2)V$ now follows from the
Arzel\`a-Ascoli theorem.
\end{proof}

\begin{lemma} \label{lem:Agmon}
Suppose $V\in \meas$ (with no support assumption) satisfies~\eqref{eq:localKato},
and for some $\lambda \not= 0$ there exists a nonzero solution $\mu \in \meas$ to
the eigenvalue equation $(I + V\Res(\lambda^2))\mu = 0$.  Then
$\Res(\lambda^2)\mu$ is an $L^2$ eigenfunction of the operator $-\Delta + V$
with eigenvalue $\lambda^2$.
\end{lemma}

\begin{proof}
Based on the one-sided inverse $(-\Delta - \lambda^2)\Res(\lambda^2) = I$ acting 
on $\meas$, any measure satisfying $\mu = -V\Res(\lambda^2)\mu$ gives rise to
the identity
\begin{equation*}
(-\Delta - \lambda^2)\Res(\lambda^2)\mu = -V\Res(\lambda^2)\mu.
\end{equation*}
Then $\Res(\lambda^2)\mu$ belongs to the null-space of
$(-\Delta + V - \lambda^2)$.  The image of a typical element of $\meas$ under
$\Res(\lambda^2)$ belongs to the weighted space $|x|^{\frac12+\eps}L^2$,
which would correspond to a resonance of $-\Delta + V$ rather than an eigenvalue.
We show next that if $\lambda \not= 0$ then in fact $\mu \in \meas$
has special mapping properties that place $\Res(\lambda^2)\mu \in L^2(\R^3)$.

Split the free resolvent into the sum of its local and nonlocal parts.  The local  part
$\Rplus_1$ is convolution against 
$(4\pi|x|)^{-1}e^{i\lambda|x|}\chi_{|x| < r}$, and the nonlocal part
$\Rplus_2$ is convolution against the bounded function
$(4\pi|x|)^{-1}e^{i\lambda|x|}\chi_{|x| \geq r}$.  The value of $r$ is chosen so that
$\sup_{y\in\R^3}  \int_{|x-y| < r} |x-y|^{-1}\,|V|(dx) < 1$.

Each solution of the eigenvalue equation satisfies $\Res(\lambda^2)\mu =
- \Res(\lambda^2)V\Res(\lambda^2)\mu$, which splits into
$-(\Rplus_1 + \Rplus_2)V\Res(\lambda^2)\mu$
leading to the identity
\begin{equation*}
\Res(\lambda^2)\mu = - (I + \Rplus_1 V)^{-1}\Rplus_2 V \Res(\lambda^2)\mu
= (I + \Rplus_1 V)^{-1}\Rplus_2 \mu
\end{equation*}

The integration kernel of $\Rplus_2$ is bounded everywhere by $r^{-1}$, so
$\Rplus_2\mu$ belongs to the space of bounded functions on $\R^3$.
Meanwhile the splitting radius was chosen so that the operator inverse
$(I + \Rplus_1 V)^{-1}$ acting on bounded functions has a convergent
Neumann series expansion.  Thus $\sup_x |\Res(\lambda^2)\mu(x)| < \infty$. 

Observe that the duality pairing $\la \Res(\lambda^2)\mu, \mu\ra$ is
well defined, with the property
\begin{equation*}
\la \Res(\lambda^2)\mu, \mu\ra  =
- \la \Res(\lambda^2)\mu, V\Res(\lambda^2)\mu\ra
= -\int_{\R^3} |\Res(\lambda^2)\mu (x)|^2 \,V(dx) \in \R
\end{equation*}
because $V$ is assumed to be a real-valued signed measure.
Consequently
\begin{equation*}
{\rm Im}\,\la \Res(\lambda^2)\mu, \mu\ra = C\lambda^{-1}
\int_{\lambda\sph^2} |\hat{\mu}(\xi)|^2\,dS(\xi)
= 0
\end{equation*}
by Parseval's identity, which implies that the Fourier transform of $\mu$
vanishes on the sphere radius $|\lambda|$.  It follows from
Proposition~4.1 of~\cite{GoSc04b} that $\Res(\lambda^2)\mu \in L^2(\R^3)$.
If $V$ is assumed to have compact support then direct examination of
$\hat{\mu}$ in the neighborhood of the sphere $|\xi| = |\lambda|$ shows that
$\Res(\lambda^2)\mu$ also has rapid polynomial decay at infinity.
\end{proof}

\begin{remark}
The proposition cited shows that $\Res(\lambda^2)f \in L^2(\R^3)$ provided
$f\in L^1(\R^3)$ and $\hat{f}\big|_{\lambda\sph^2} = 0$.
It can be extended to finite measures with only superficial changes to the
proof.
\end{remark}

Lemma~\ref{lem:Agmon} eliminates the possibility of embedded resonances.
Pointwise existence of $(I + V\Res(\lambda^2))^{-1}$ still requires that
$-\Delta + V$ have no embedded eigenvalues, and no resonance or
eigenvalue at $\lambda = 0$.  These properties are incorporated into the
spectral assumptions of Theorem~\ref{thm:main}.  For a large class of
potentials the embedded eigenvalue condition is automatically satisfied.

\begin{proposition} \label{prop:absence}
Suppose $V$ is a compactly supported measure of dimension $d > 2$.
Then $-\Delta + V$ has no embedded eigenvalues $\lambda^2 > 0$.
\end{proposition}

\begin{proof}
The unique continuation theorems in~\cite{KoTa06} asserting the absence of
embedded eigenvalues apply here provided multiplication by $V$ is a bounded
map from $\dot{W}^{\frac14, 4}(\R^3)$ to its dual space.  An equivalent condition is
that $\dot{W}^{\frac14,4}(\R^3)$ embeds as a subspace of $L^2(V)$.

If $V$ is supported in $B(0,2^M)$ with dimesnion $d > 2$ then $V$ also
satisfies the Kato-type condition $\sup_y \int_{\R^3} |x-y|^{-\gamma}\,|V|(dx)
< \infty$ for any $\gamma <d$,
by imitating the proof of Proposition~\ref{prop:dimensionKato}.
Applying the argument for Kato-class potentials  in Section~\ref{sec:selfadjoint} leads
to the conclusion here that $(-\Delta)^{\frac{\gamma-3}{2}}V$ is a bounded map on
$L^p(V)$ for all $1 \leq p \leq \infty$, and in particular that
$\dot{H}^{\frac{3-\gamma}{2}}(\R^3)$ embeds into $L^2(V)$.  The same is
certainly true for the non-homogeneous space $H^{\frac{3-\gamma}{2}}$ as well.

A straightforward $T^*T$ argument shows that 
$(1-\Delta)^{\frac{2-\gamma}{2}}L^\infty(\R^3)$ is also contained in $L^2(V)$
for any $\gamma > 2$.  Noting that $(1-\Delta)^{\frac{2-\gamma}{2}}$ has
an integrable convolution kernel $K_\gamma(x)$ there is an immediate bound
\begin{equation*}
\norm[(1-\Delta)^{\frac{2-\gamma}{2}}V(1-\Delta)^{\frac{2-\gamma}{2}}f][1]
 \leq \norm[K_\gamma][1]^2 \norm[V][\meas] \norm[f][\infty].
\end{equation*}

Finally, the fact that $W^{\frac14,4}(\R^3) \subset L^2(V)$ follows by choosing
$2 < \gamma < d$ and applying Riesz-Thorin interpolation to the $L^2$ and
$L^\infty$ estimates.  The assumption that $V$ is compactly supported permits
an extension to the homogeneous Sobolev space $\dot{W}^{\frac14,4}(\R^3)$ as
desired.

\end{proof}

At this point we have shown that 
$\norm[(I + \eta(\lambda/L)V\Res(\lambda^2))^{-1}][]$
is uniformly bounded on any compact set $\lambda \in [-L_0, L_0]$ for all
$L \geq L_0$.  A separate high-energy argument is required to set a
value for $L_0$ for which the operator inverse can be controlled independent
of $|\lambda|, L > L_0$.  We will show that, under the assumption that $V$ has
dimension $d>\frac32$, that 
$\lim_{\lambda\to\pm\infty}\norm[(V\Res(\lambda^2))^2][] = 0$ as a bounded
operator on $\meas$.
Then the operator inverse is controlled by $1 + \norm[V][\Kato]$ for all
sufficiently large $\lambda$ and $L$, by applying~\eqref{eq:T_LinW_M} and
summing the Neumann series.

Our calculations rely on a resolvent estimate at high energy relating
$L^2(V)$ to its dual space.

\begin{theorem} \label{thm:resolvent}
Suppose $V$ is a compactly supported measure of dimension $d > \frac32$.
There exists $\eps > 0$ so that the free resolvent satisfies
\begin{equation} \label{eq:resolvent}
\norm[\Res(\lambda^2)Vf][L^2(V)] \les \japanese[\lambda]^{-\eps}
\norm[f][L^2(V)].
\end{equation}
\end{theorem}

There are close connection between the free resolvent $\Res(\lambda^2)$
and the restriction of Fourier transforms to the sphere $\lambda\sph^2$.
We make use of a Fourier restriction estimate proved by Erdogan~\cite{Er05},
with the specific case of interest in three dimensions extracted below.
\begin{theorem}[\cite{Er05}, Equation (15)] \label{thm:Erdogan}
Let $A_R$ denote the annulus $|x-R|< 1$ inside $\R^3$, with $R>1$.  
Suppose $V$ is a compactly supported measure of dimension 
$d \in (\frac32, \frac52)$.
Then functions $g$ supported in $A_R$ satisfy an inequality
\begin{equation} \label{eq:Erdogan}
\norm[g^\vee][L^2(V)] \les R^{\beta}\norm[g][2]
\end{equation}
for each $\beta > \frac78 - \frac{d}{4}$.  In particular it is possible to choose
$\beta < \frac12$.
\end{theorem}

\begin{proof}[Proof of Theorem~\ref{thm:resolvent}]
The specific inequality we derive has the form
\begin{equation*}
\norm[\Res(\lambda^2)Vf][L^2(V)] \les \lambda^{2\beta-1}\log \lambda
\norm[f][L^2(V)]
\end{equation*}
uniformly over $\lambda > 4$.  The logarithmic factor is most likely an
artifact of the method of estimation.  The case $\lambda < -4$ is identical
up to complex conjugation.

The free resolvent $R^\pm_0(\lambda^2)$ acts by multiplying Fourier transforms
pointwise by the distribution
\begin{equation*}
\frac{1}{|\xi|^2 - \lambda^2} 
\pm i\frac{\pi}{\lambda}\,d\sigma(|\xi| = |\lambda|).
\end{equation*}
For the surface measure term it suffices to note that since $V$ has compact
support the dual statement to~\eqref{eq:Erdogan} implies that
\begin{align*}
\bignorm[(Vf)^\wedge][L^2(A_R)] &\les
R^\beta \norm[f][L^2(V)] \\
{\rm and} \quad \bignorm[\nabla_\xi (Vf)^\wedge][L^2(A_R)] &\les
R^\beta \norm[f][L^2(V)]
\end{align*} 
from which it follows that $\bignorm[(Vf)^\wedge\big|_{|\xi| = R}][L^2(d\sigma)]
\les R^\beta \norm[f][L^2(V)]$.  The same estimates hold for higher derivatives 
$D^\alpha_\xi(Vf)^\wedge$ by considering $x^\alpha f \in L^2(V)$
instead of $f$. In particular there is control of the outward normal gradient 
\begin{equation*}
\Bignorm[{\txt \frac{\xi}{|\xi|}}
 \cdot\nabla_\xi(Vf)^\wedge(\xi)\big|_{|\xi|=R}][L^2(d\sigma)]
\les R^\beta \norm[f][L^2(V)]
\end{equation*}
which will come in handy in the next step.

Let $\phi$ be a smooth function
supported in the annulus $\frac12 \leq |\xi|< 2$ that is identically 1 when
$\frac34 \leq |\xi| \leq \frac32$. First use $\phi$ to cut the Fourier multiplier away
from the sphere of radius $\lambda$, with the result
\begin{equation*}
K_\lambda (x) := \Big(\frac{1-\phi(\xi/\lambda)}{|\xi|^2 - \lambda^2}\Big)^\vee (x) \sim
\begin{cases}
|x|^{-1}\ &{\rm if}\ |x| < \lambda^{-1} \\
\lambda\, O(\lambda |x|)^{-N} &{\rm if} |x| \geq \lambda^{-1}
\end{cases}
\end{equation*}

A slightly modified version of~\eqref{eq:dimensionKato} shows that
$\int_{\R^3} |K_\lambda(x-y)|\,|V|(dx) \les \lambda^{1-d}$ uniformly
in $y$.
The Schur test 
then shows that integration against
$K_\lambda(x-y)V(y)$ defines a bounded operator on $L^p(V)$, 
$1 \le p \le \infty$, with norm comparable to $\lambda^{1-d}$.
This part of the free resolvent, with frequencies removed from $\lambda$,
is bounded on $L^2(V)$ and enjoys relatively rapid polynomial decay
since $d > \frac32$.

For each $\frac{\lambda}{2}< s <2\lambda$ define $F_s(x)$ to
be $Vf * \frac{\sin(s|x|)}{|x|}$ so that
$s \hat{F}_s(\xi)$ is the restriction of $4\pi (Vf)^\wedge(\xi)$ to the sphere
$|\xi| =s$.  Based on the preceding estimates, both $F_s$ and $\frac{d}{ds}F_s$
belong to $L^2(V)$ with norms bounded by $\lambda^{2\beta - 1}$ uniformly
over the interval $s \sim \lambda$.

The remaining part of the free resolvent appears as a principal value integral,
with the derived bound
\begin{equation*}
\Bignorm[p.v. \int_{\lambda/2}^{2\lambda} 
\Big(\frac{s\phi(\frac{s}{\lambda})}{s+\lambda} F_s\Big) 
 \frac{1}{s-\lambda}\,ds][L^2(V)] \les \lambda^{2\beta - 1}\log \lambda.
\end{equation*} 
The size and smoothness of $F_s$ make it possible to bring the norm inside when
$|s-\lambda| > 1$, and to lessen the singularity via integration by parts
when $|s-\lambda| \leq 1$.
\end{proof}

\begin{corollary} \label{cor:Resdecay}
Suppose $V$ is a compactly supported measure of dimension $d > \frac32$.
Then there exists $\eps > 0$ so that
\begin{equation} \label{eq:Resdecay}
\norm[(V\Res(\lambda^2))^k \mu][\meas] \les \japanese[\lambda]^{-\eps(k-1)}
C(V)^k \norm[\mu][\meas]
\end{equation}
\end{corollary}

\begin{proof}
The convolution kernel of $\Res(\lambda^2)$ is dominated by $|x-y|^{-1}$,
which belongs to $L^p(V)$ uniformly over $y \in \R^3$ for each $1\le p < d$.
Then $\Res(\lambda^2)$ maps $\meas$ to $L^p(V)$.  Interpolation between
Theorem~\ref{thm:resolvent} and the elementary $L^1(V)$ bounds yields
an operator bound
\begin{equation*}
\norm[\Res(\lambda^2)V f][L^p(V)] \le 
C(V)\japanese[\lambda]^{-\frac{2}{p'}\eps} \norm[f][L^p(V)]
\end{equation*}
which can be applied $(k-1)$ times, followed by an inclusion map
$L^p(V) \hookrightarrow L^1(V)$.
\end{proof}

For our application,
choose $L_0$ large enough so that~\eqref{eq:Resdecay} ensures the
operator norm of $(V\Res(\lambda^2))^2$ is less than $\frac12$ for all
$|\lambda| > L_0$.  Then
\begin{equation*}
\bignorm[(I + \eta(\lambda/L)V\Res(\lambda^2))^{-1}][\meas\to\meas]
< C(1 + \norm[V][\Kato])
\end{equation*}
for all $\lambda, L > L_0$.  We previously showed that
$(I + V\Res(\lambda^2))^{-1}$ is bounded and continuous over the interval
$\lambda \in [-L_0, L_0]$, and the introduction of a cutoff $\eta(\lambda/L)$
has no effect there once $L > L_0$.  The combined bounds show that
\begin{equation*}
\alpha_L := \sup_{\lambda \in \R}
\norm[({\mathbf 1} + \hat{T}_L(\lambda))^{-1}][\meas\to\meas]
\le \alpha < \infty
\end{equation*}
uniformly for $L > L_0$.

The choice of $N$ is governed by Corollary~\ref{cor:Resdecay}.
Set $N$ to be the first integer large enough so that $(N-1)\eps > 2$.
This number depends only on the dimension of $V$ without regard to
any measures of its size.

The next parameter to consider is $R_L$, which governs the inequality
\begin{equation*}
\norm[\chi_{|\rho|\ge R_L} T_L][\W_\meas] \leq (K\alpha)^{-1}
\end{equation*}
This is controlled by direct examination of the integral kernel of $T_L$
in~\eqref{eq:K_L}.  More specifically,
\begin{equation} \label{eq:R_L}
\begin{aligned}
\norm[\chi_{|\rho|\ge R} T_L][\W_\meas] &= 
\sup_{y\in\R^3} \int_{\R^3}\int_{|\rho| > R} 
L\frac{|\check{\eta}(L(\rho+|x-y|))|}{4\pi|x-y|}\, d\rho |V|(dx) \\
&\les \sup_{y\in\R^3} \int_{|x-y| > R-1} \frac{|V|(dx)}{|x-y|}
+ L^{-1} \int_{|x-y| < R-1} \frac{|V|(dx)}{|x-y|}
\end{aligned}
\end{equation}
The first integral is bounded by $R^{-1}\norm[V][\meas]$ for any $R > 2$.
Meanwhile the second integal
is less than $L^{-1}\norm[V][\Kato]$ so for all $L \gtr \alpha\norm[V][\Kato]$
it suffices to choose $R$ to be the larger of $2$ and $C\alpha\norm[V][\meas]$.

It will be convenient in the next step to have additional control of 
$\norm[\chi_{|\rho| \ge R} T_L][\W_\meas]$.  Using the same construction, one can
find $R$ so that
\begin{equation*}
\norm[\chi_{|\rho|\ge R} T_L][\W_\meas] \leq 
  (KN)^{-1}(C \norm[V][\Kato])^{-(N-1)}
\end{equation*}
for any $L \gtr  N (C \norm[V][\Kato])^N$.

The purpose of $N$ is to give $\norm[(\hat{T}_L)^N(\lambda)][]$ sufficiently
rapid decay (in $\lambda$) so that Fourier inversion forces
$(T_L)^N(\rho)$ to be a continuously differentiable operator-valued
function, with derivative smaller in operator norm  than
$C(V)^N$.  For each fixed $0< \delta < 1$ one can estimate the size
of the difference $(T_L)^N(\rho) - (T_L)^N(\rho-\delta)$ using the mean
value theorem with the result
\begin{align*}
\norm[(T_L)^N(\,\cdot\,) - (T_L)^N(\,\cdot\,-\delta)][\W_\meas]
&\leq \int_{|\rho|\leq NR + 1} 
 \norm[(T_L)^N(\rho)-(T_L)^N(\rho-\delta)][{\mathcal B}(\meas)]\, d\rho \\
&\hskip 1.5in + 2 \norm[\chi_{|\rho| \geq NR}(T_L)^N][\W_\meas]\\
& \leq \delta (C(V))^N NR + 2N\norm[\chi_{|\rho|\ge R} T_L][\W_\meas]
 \norm[T_L][\W_\meas]^{N-1}
\end{align*}
Based on the prior estimates the choice of $\delta < (NRK)^{-1}C(V)^{-N}$
will cause
\begin{equation*}
 \norm[T_L(\,\cdot\,) - T_L(\,\cdot\,-\delta)][\W_\meas]
 \le 3/K
\end{equation*}
for all $L > \max(L_0, \alpha\norm[V][\Kato], N (C \norm[V][\Kato])^N)$.
This concludes the proof of Theorem~\ref{thm:main} in the case where
$V$ is a compactly supported measure of dimension $d > \frac32$.

The extension of dispersive estimates to potentials $V$ which are the Kato-norm
limit of compactly supported $d$-dimensional measures is more or less routine.
As before the norm of $T_L$ within $\W_\meas$ has a uniform bound in terms
of $\norm[V][\Kato]$ from~\eqref{eq:T_LinW_M}.  No approximation properties
are required in this step.

The arguments used to establish a finite value for $\alpha$ require more individual
attention.  Proposition~\ref{prop:dimensionKato} guarantees that $V$
satisfies~\eqref{eq:localKato} so there are no complications regarding the
self-adjointness of $-\Delta + V$.  Elementary limiting arguments can be applied
to~\eqref{eq:continuity} to show that $V\Res(\lambda^2)$ has continuous
(but not necessarily Lipschitz) dependence on $\lambda$, and to
Lemma~\ref{lem:compactness} to show that each operator $V\Res(\lambda^2)$
is compact.   Under the assumption that $-\Delta + V$ has no resonance at zero
and no threshold or embedded eigenvalues, there is a uniform norm bound on
$(I + V\Res(\lambda^2))^{-1}$ over any finite interval
$\lambda \in [-L_0, L_0]$.  Note that embedded resonances are still forbidden by
Lemma~\ref{lem:Agmon}.

For the low and intermediate energy estimates, it suffices for $V$ to be the limit
(in $\Kato$) of a sequence of compactly supported measures
satisfying~\eqref{eq:localKato}.  Following Corollary~\ref{cor:Resdecay}, if
each $\mu_j$ in the approximating sequence has dimension $d_j > \frac32$
then
\begin{equation*}
\lim_{\lambda \to \pm \infty} \norm[(V\Res(\lambda^2))^2][\meas \to \meas] = 0.
\end{equation*}
Then $L_0$ can be set large enough for the Neumann series of
$(I + \eta(\lambda/L)V\Res(\lambda^2))^{-1}$ to converge uniformly over
the infinite intervals $|\lambda| > L_0$.  Combined with the low-energy
results this gives a finite bound for $\alpha_L$ once $L > L_0$.

The process for choosing $N$ requires that $d = \inf d_j > \frac32$, so that
there is a uniform value of $\eps>0$ in Corollary~\ref{cor:Resdecay}. Then one can
again declare $N$ to be the smallest integer satisfying $(N-1)\eps > 2$.

The selection criteria for $R_L$ are little changed.  So long as $V$ is the
Kato norm-limit of compactly supported potentials it is permissible to 
estimate
\begin{equation*}
\lim_{R \to \infty} \sup_{y\in\R^3} \int_{|x-y|>R-1} \frac{|V|(dx)}{|x-y|} = 0
\end{equation*}
inside of~\eqref{eq:R_L}, replacing the explicit decay rate of
$R^{-1}\norm[V][\meas]$.  Thus for some $R < \infty$ this integral term will be
smaller than $(K\alpha)^{-1}$.  The lower bounds placed on $L$ are independent
of the support of $V$ and do not need further adjustment.

To find $\delta$, first choose an approximating measure $\mu$ of dimension
$d > \frac32$ such that
$\norm[V-\mu][\Kato] \les (KN)^{-1}\norm[V][\Kato]^{-(N-1)}$.  With $\tilde{T}_L$
denoting the element of $\W_\meas$ generated by the potential $\mu$ and cutoff
$\eta(\lambda/L)$, it follows from~\eqref{eq:T_LinW_M} that
\begin{equation*}
\norm[T_L - \tilde{T}_L][\W_\meas] < (KN)^{-1}\norm[V][\Kato]^{-(N-1)}
\end{equation*}
uniformly for all $L>0$.  Then the norm difference of their respective
$N^{\rm th}$ powers is controlled by $ K^{-1}$.  Based on the properties of
$\mu$, one can choose $\delta > 0$ so that 
$\norm[(\tilde{T}_L)^N(\,\cdot\,) - 
  (\tilde{T}_L)^N(\,\cdot\,-\delta)][\W_\meas] \leq 3/K$
for all sufficiently large $L$.
By the triangle inequality a similar translation bound holds for $T_L$ as well, with
\begin{equation*}
\norm[(T_L)^N(\,\cdot\,) - (T_L)^N(\,\cdot\,-\delta)][\W_\meas] \leq 5/K.
\end{equation*}

All five of the parameters related to the construction of $(\1 + T_L)^{-1}$
(namely $\norm[T_L][\Kato]$, $\alpha_L $, $R$, $N$, and $\delta$) have an
eventual uniform bound as $L \to \infty$. Therefore the operator inverses
$\norm[(\1 + T_L)^{-1}][\overline{W}_\meas]$
are also uniformly bounded, and their weak limit $(\1 + T)^{-1}$ is a
finite measure on $\R^{1+3}$ as desired.

\bibliographystyle{abbrv}
\bibliography{MasterList}

\end{document}